\documentclass[reqno,12pt]{amsart}
\usepackage{amsfonts,amsmath,amsthm,amssymb,amscd,verbatim}
\usepackage[all]{xy}
\usepackage{xcolor}
\newcommand{\version}{Ver.~0.0}
\newcommand{\setversion}[1]{\renewcommand{\version}{Ver.~{#1}}}
\setversion{0.0 [2014/06/20 09:42:45]}
\setversion{0.9 [2014/06/25]}
\setversion{1.0 [2014/06/29 15:11:16]}
\setversion{1.01 [2014/06/30 23:42:50]}
\setversion{1.02 [2014/07/01 14:36:46]}
\setversion{1.5 [2014/09/15 17:06:24]}
\setversion{1.6 [2014/11/23 15:49:03]}

\title [Enhanced Orbit Embedding]
{Enhanced orbit embedding}

\author[K.~Nishiyama]{Kyo Nishiyama$^{\ast}$}
\address{
Department of Physics and Mathematics\\
Aoyama Gakuin University\\
Fuchinobe 5-10-1, Chuo-ku, Sagamihara 252-5258, Japan}
\email{kyo@gem.aoyama.ac.jp}

\thanks{$\ast)${\ } Supported by JSPS Grant-in-Aid for Scientific Research \#{25610008}.}

\dedicatory{Dedicated to Professor Fumihiro Sato on the occasion of his 65th birthday}

\date{\version\quad(compiled on \today)}

\theoremstyle{plain}
\newtheorem{theorem}{Theorem}
\newtheorem{proposition}[theorem]{Proposition}
\newtheorem{corollary}[theorem]{Corollary}
\newtheorem{lemma}[theorem]{Lemma}

\newtheorem{assumption}[theorem]{Assumption}
\theoremstyle{definition}

\theoremstyle{remark}
\newtheorem{remark}[theorem]{\upshape Remark}

\numberwithin{equation}{section}
\numberwithin{theorem}{section}

\newcommand{\Z}{\mathbb{Z}}

\newcommand{\R}{\mathbb{R}}

\newcommand{\C}{\mathbb{C}}

\newcommand{\lie}[1]{\mathfrak{#1}}

\newcounter{thmenum}
\newenvironment{thmenumerate}{%
\begin{list}{$(\thethmenum)$}{%
\usecounter{thmenum}
\setlength{\labelsep}{.5em}
\setlength{\labelwidth}{-7pt}
\setlength{\topsep}{0pt}
\setlength{\partopsep}{0pt}
\setlength{\parsep}{0pt}
\setlength{\leftmargin}{3pt}
\setlength{\rightmargin}{0pt}
\setlength{\itemindent}{\leftmargin}
\setlength{\itemsep}{0pt}
}}
{\end{list}}

\newcommand{\mycomment}[1]{} 

\newlength{\lengthcup}
\newcommand{\diag}{\qopname\relax o{diag}}

\newcommand{\End}{\qopname\relax o{End}}

\newcommand{\id}{\qopname\relax o{id}}

\newcommand{\Aut}{\qopname\relax o{Aut}}
\newcommand{\Ad}{\qopname\relax o{Ad}}

\renewcommand{\Re}{\qopname\relax o{Re}}

\newcommand{\closure}[1]{\overline{#1}}

\newcommand{\transpose}[1]{\,{}^t{#1}}

\newcommand{\restrict}{\big|}

\newcommand{\Spec}{\mathop\mathrm{Spec}\nolimits{}}
\newcommand{\Sym}{\mathop{\mathrm{Sym}}\nolimits}

\newcommand{\orbit}{\mathbb{O}}

\newcommand{\GITquotient}{/\!/}
\newcommand{\git}{\GITquotient}

\newcommand{\GL}{\mathrm{GL}}

\newcommand{\Sp}{\mathrm{Sp}}

\newcommand{\Mat}{\mathrm{M}}

\newcommand{\gl}{\lie{gl}}

\newcommand{\inclusion}{\hookrightarrow}

\newcommand{\tildeG}{\widetilde{G}}
\newcommand{\tildeK}{\widetilde{K}}
\newcommand{\tildeg}{\widetilde{\lie{g}}}
\newcommand{\tildek}{\widetilde{\lie{k}}}
\newcommand{\tildep}{\widetilde{\lie{p}}}
\newcommand{\tildeL}{\widetilde{L}}
\newcommand{\tildeX}{\widetilde{X}}

\newcommand{\mattwo}[4]{\Bigl(\begin{array}{@{\,}c@{\;\;}c@{\,}}{#1} & {#2} \\ {#3} & {#4} \end{array}\Bigr)}

\begin{document}

\begin{abstract} 
Let $ \tildeG $ be an algebraic group acting on a variety $ \tildeL $, and   
$ G \subset \tildeG $ a subgroup which leaves a subvariety $ L \subset \tildeL $ stable.  
For a $ G $-orbit $ \orbit_G = G \, u \;\; (u \in L) $ in $ L $, we can associate an orbit 
$ \orbit_{\tildeG} = \tildeG \, u $ of $ \tildeG $ so that we get a map 
$ L/G \to \tildeL/\tildeG $ between orbit spaces, 
though this map is usually not injective.  
In this note, when $ G $ is a symmetric subgroup arising from an involutive anti-automorphism, 
we give certain sufficient conditions for the map $ L/G \to \tildeL/\tildeG $ to be injective  
after the method of Ohta \cite{Ohta.2008}.  
Our main concern here is to produce examples of enhanced Lie algebras (or enhanced $ \theta $-representations).  
We also analyze an obstruction which prevents the orbit space inclusion.
\end{abstract}

\maketitle

\section{Introduction}

Let $ \tildeG $ be an algebraic group acting on a variety $ \tildeL $.  
We denote by $ \tildeL/\tildeG $ the orbit space (the collection of all orbits) with quotient topology.  
Consider a subgroup $ G \subset \tildeG $ and let us assume that 
a subvariety $ L \subset \tildeL $ is stable under $ G $.  
Take a $ G $-orbit $ \orbit_G = G \, u \;\; (u \in L) $, 
and we associate it with a $ \tildeG $-orbit $ \orbit_{\tildeG} = \tildeG \, u $.  
Thus we get a continuous map $ L/G \to \tildeL/\tildeG $, 
though this map is usually not injective.  

In the case where $ \tildeL $ is a vector space with linear $ \tildeG $-action, 
Ohta \cite{Ohta.2008} studied embedding of $ G $-orbits in $ L $ into $ \tildeG $-orbits 
with application to the invariant theory in mind (see also \cite{Ohta.2010}).  
Ohta put some reasonable conditions to get the injectivity 
of the map $ L/G \inclusion \tildeL/\tildeG $.  
It seems that his conditions can be somewhat loosened so that we can manage much more interesting cases.

We generalize Ohta's conditions in two ways.  
One is to replace a vector space $ \tildeL $ (in Ohta's case) 
by a general variety.  
In application, this is important.  
The other one is to replace one of Ohta's conditions which depends on the realization of $ \tildeG $ 
and seemingly its ambient space.  
This replaced condition is actually very close to his original one, 
and seems to catch an important property of orbits in an abstract way.  
We also would like to emphasize Ohta's original method provides a very efficient tool to check 
this new condition.

The main theorem is stated in \S~\ref{setion:main.theorem}, however, 
it is helpful to explain our main example here.  

We take $ \tildeG = \GL_{2n}(\C) $ which acts on $ V = \C^{2n} $ by the matrix multiplication.  
Then $ \tildeG $ acts on $ \tildeL = V \oplus V^{\ast} \oplus \gl_{2n}(\C) $ in
natural way.  This is our \emph{enhanced Lie algebra}.
Let us consider a subgroup $ G = \Sp_{2n}(\C) $ acting on $ L = V \oplus \lie{sp}_{2n}(\C) $, 
which is embedded into $ \tildeL $ using $ V \inclusion V \oplus V^{\ast} $ as 
$ u \mapsto (u, - \transpose{u}) $.  For the Lie algebra part, we use a natural embedding as a subalgebra.  
Note that $ L $ is also an enhanced Lie algebra for $ \lie{sp}_{2n}(\C) $.  
Then, our theorem tells

\begin{theorem}
Natural map $ L/G \to \tildeL / \tildeG $ induces an embedding of orbit spaces
\begin{equation*}
L/ G = \Bigl( V \oplus \lie{sp}_{2n}(\C)  \Bigr) \Big/ \Sp_{2n}(\C) 
\inclusion 
\Bigl( V \oplus V^{\ast} \oplus \gl_{2n}(\C) \Bigr) \Big/ \GL_{2n}(\C) = \tildeL / \tildeG .
\end{equation*}
In particular, if we take a $ \GL_{2n}(\C) $-orbit $ \orbit_{\GL_{2n}} \subset \tildeL $, 
then the intersection $ \orbit_{\GL_{2n}} \cap L $ consists of a single $ \Sp_{2n}(\C) $-orbit.
This embedding takes closed $ G $-orbits to closed $ \tildeG $-orbits.
\end{theorem}


In the end of Section~\ref{section:enhanced.Lie.algebra}, 
we will discuss on an application to invariant theory.

In the course of arguments, 
we show another new example where a restriction of this orbit embedding is still an embedding.
Let $ \tildeK = \GL_n(\C) \times \GL_n(\C) $ which is diagonally embedded into $ \tildeG = \GL_{2n}(\C) $.  
Then $ K = G \cap \tildeK $ is isomorphic to $ \GL_n(\C) $.  
Let $ \Mat_n(\C) \oplus \Mat_n(\C) $ be anti-diagonally embedded into $ \gl_{2n}(\C) $, 
which is considered to be $ \tildeg / \tildek $ 
(the tangent space of $ \tildeG / \tildeK $ at the base point).  
Similarly, we see that $ \Sym_n(\C) \oplus \Sym_n(\C) $ is anti-diagonally embedded into 
$ \lie{sp}_{2n}(\C) \subset \gl_{2n}(\C) $, 
which can be identified with $ \lie{g}/\lie{k} \simeq T_{eK} (G/K) $.  
In this setting we have

\begin{corollary}
Under the above setting, 
there is an embedding of orbit space
\begin{equation*}
\begin{aligned}
& \Bigl( \C^n \oplus \Sym_n(\C) \oplus \Sym_n(\C) \Bigr) \Big/ \GL_n(\C) 
\\
& \qquad \qquad
\inclusion 
\Bigl( \C^{2n} \oplus \Mat_n(\C) \oplus \Mat_n(\C) \Bigr) \Big/ \GL_n(\C) {\times} \GL_n(\C) .
\end{aligned}
\end{equation*}
\end{corollary}

This corollary should be interpreted as an orbit embedding for 
\emph{enhanced $ \theta $-repre\-sentations}.  
Note that the actions of $ K = \GL_n(\C) $ and $ \tildeK = \GL_n(\C) \times \GL_n(\C) $ are somewhat twisted, 
and also the embedding is twisted too.  
See \S~\ref{sec:further.examples} for details.

\medskip

Our treatment here provides a useful method to determine whether 
a particular $ \tildeG $-orbit intersects $ L $ in a single $ G $-orbit.  
In the last section, we exhibit this by using a simple example, where $ \tildeG = \GL_2(\C) $.

\medskip

\textbf{Acknowledgement.}  
The author thanks the referee for suggestions which are of great help to improve this note.  
Also he is indebted to the referee for pointing out a serious error in the first draft. 
The author also thanks Takuya Ohta, Lucas Fresse, Hiroyuki Ochiai and Minoru Itoh for useful discussions.  
Anthony Henderson kindly informed the author that $ \gamma_k $ in 
Corollary~\ref{corollary:generators.of.invariants.for.SpV} 
vanishes when $ k $ is even.  

The collaboration with Lucas is a good motivation to this short note, and 
the on-going joint works with him will appear elsewhere.

Last but not least, the author expresses 
a deep admiration to Fumihiro-sensei for his contributions to mathematics.

\section{Enhanced Lie algebra and the orbit embedding}
\label{section:enhanced.Lie.algebra}

Let $ V = \C^{2 n} $ be a vector space with symplectic form $ \langle , \rangle $ defined by 
\begin{equation}
\langle u, v \rangle = \transpose{u} J v , 
\qquad
\text{where } 
J = \Bigl( \begin{array}{r|r}
 & \;\;1_n \\ \hline
-1_n & 
\end{array} \Bigr) .
\end{equation}
We put $ \tildeG = \GL(V) = \GL_{2n}(\C) $ and 
consider the following anti-automorphism $ \sigma $ of $ \tildeG $:
\begin{equation}
\sigma(g) = J^{-1} \transpose{g} J 
\qquad
(g \in \tildeG).  
\end{equation}
This means that $ \sigma(g) = g^{\ast} $, the adjoint of $ g $ with respect to the symplectic form $ \langle , \rangle $.  
We naturally extend it to $ \End(V) $ and denote it by the same letter $ \sigma $.
Note that $ \sigma $ is involutive, i.e., $ \sigma^2 = \id $.  
Thus $ \theta(g) = \sigma(g)^{-1} $ defines an involutive automorphism of $ \tildeG $, 
and we put 
\begin{equation}
G := G^{\theta} 
= \{ g \in \tildeG \mid \theta(g) = g \} 
= \{ g \in \tildeG \mid \sigma(g) = g^{-1} \} , 
\end{equation}
the symplectic group $ \Sp_{2n}(\C) $, 
which is a symmetric subgroup of $ \tildeG $.  
Our main concern here is an action of $ \tildeG $ on an enhanced Lie algebra
\begin{equation}
\tildeL := V \oplus V^{\ast} \oplus \End(V) 
\end{equation}
with the action explicitly given as follows: 
for $ X = (u, \transpose{v}) + A \in V \oplus V^{\ast} \oplus \End(V) $ and $ g \in \tildeG $, 
\begin{equation}
g \cdot X 
= (g u , \transpose{(\sigma(g)^{-1} v)}) + g A g^{-1}
= (g u , \transpose{(\theta(g) v)}) + \Ad(g) A .
\end{equation}
Note that $ \End(V) $ can be considered as a Lie algebra $ \tildeg = \gl_{2 n}(\C) $ with the adjoint action.
The action is somewhat twisted by $ \sigma $, but actually, as a representation, the space $ \tildeL $ is isomorphic to 
$ V \oplus V^{\ast} \oplus V \otimes V^{\ast} $ in usual sense.
The space $ \tildeL $ is our enhanced Lie algebra.

We define an involutive automorphism of a vector space $ \tildeL $ as follows, and 
we will denote it by the same letter $ \sigma $.  
Namely, for $ X = (u, \transpose{v}) + A \in \tildeL $, we put 
\begin{equation}
\sigma(X) = \sigma((u, \transpose{v})) + \sigma(A) 
:= (v , \transpose{u}) + J^{-1} \transpose{A} J .
\end{equation}
Since this involution is a natural extension from an anti-automorphism of the matrix algebra $ \End(V) $, 
there should be no confusion to use the same letter.  
However, if there is a possibility of confusion, we will use $ \sigma_{\tildeG} $ or $ \sigma_{\tildeL} $ depending on the situation.
So we have $ \sigma_{\tildeL}^2 = \id_{\tildeL} $, and note that $ \sigma_{\tildeL} $ is an automorphism of a \emph{vector space}.

\begin{lemma}
\label{lemma:sigma.twists.tildeG.action}
For $ g \in \tildeG $ and $ X \in \tildeL $, 
we have 
$ \sigma(g \cdot X) = \sigma(g)^{-1} \cdot \sigma(X) $.
\end{lemma}

\begin{proof}
This is a consequence of the following calculation:
\begin{align*}
\sigma(g \cdot X) 
&= \sigma((g u, \transpose{(\sigma(g)^{-1} v)})) + \sigma(g A g^{-1}) 
\\
&= (\sigma(g)^{-1} v, \transpose{(g u)}) + \sigma(g)^{-1} \sigma(A) \sigma(g) 
\\
&= \sigma(g)^{-1} \cdot \bigl( (v, \transpose{u}) + \sigma(A) \bigr) 
= \sigma(g)^{-1} \cdot \sigma(X) .
\end{align*}
\end{proof}

Let us consider another automorphism $ \alpha \in \Aut_{\tildeG}(\tildeL) $ of $ \tildeL $ which is $ \tildeG $-equivariant.  
In the present case, $ \alpha $ should be non-zero constant on each irreducible components of $ \tildeL $, so 
there is not much possibility.  In the application, we will take $ \alpha $ as $ -1 $ or $ 1 $.
Put
\begin{equation}
 L := \{ X \in \tildeL \mid \sigma(X) = \alpha(X) \}.
\end{equation}

\begin{lemma}
The subspace $ L $ is stable under the action of $ G $.
\end{lemma}

\begin{proof}
Take $ X \in L $ so that $ \sigma(X) = \alpha(X) $ holds.  
For $ g \in G $, we have
\begin{align*}
\sigma(g \cdot X) &= \sigma(g)^{-1} \cdot \sigma(X) = g \cdot \sigma(X) 
\intertext{by Lemma~\ref{lemma:sigma.twists.tildeG.action}, and }
\alpha(g \cdot X) &= g \cdot \alpha(X) ,
\end{align*}
since $ \alpha $ is $ \tildeG $-equivariant.  
Thus we get $ \sigma(g \cdot X) = \alpha(g \cdot X) $ which means $ g \cdot X \in L $.
\end{proof}

We are interested in the comparison of two orbit spaces $ L/G $ and $ \tildeL / \tildeG $.  
Before proceeding further, it will be helpful to sketch what is our $ L $.  
We have a decomposition $ \tildeg = \lie{g} \oplus \lie{p} $ according to the eigenvalue $ \pm 1 $ of $ \theta $ (polar decomposition).  
Then, if we take $ \alpha = -1 $, we get
\begin{equation}
L = \{ X = (u, - \transpose{u}) + A \mid A \in \lie{g} \} \simeq V \oplus \lie{sp}(V) .
\end{equation}
This is an enhanced Lie algebra for $ G $.  
On the other hand, if we take $ \alpha = 1 $ (identity map), then we get 
\begin{equation}
L = \{ X = (u, \transpose{u}) + A \mid A \in \lie{p} \} \simeq V \oplus \wedge^2 V
\end{equation}
and this is an enhanced $ \theta $-representation (cf.~\cite{Vinberg.1976}, \cite{Graaf.Vinberg.Yakimova.2012}), 
or an exotic symplectic Lie algebra (\cite{Kato.2009}).  

Now we present a main theorem under the setting above.

\begin{theorem}
\label{theorem:orbit.embedding.for.enhanced.Lie.algebra}
The natural correspondence of orbits of $ G $ and $ \tildeG $ given by 
\begin{equation*}
L/G \ni \orbit_G \longmapsto \tildeG \cdot \orbit_G \in \tildeL / \tildeG 
\end{equation*}
induces an injection.  
In other words, 
for any $ \tildeG $-orbit $ \orbit_{\tildeG} \subset \tildeL $, the intersection 
$ L \cap \orbit_{\tildeG} $ is either empty or a single $ G $-orbit.
This injection takes closed $ G $-orbits to closed $ \tildeG $-orbits.
\end{theorem}

\begin{proof}
Take $ X, Y \in \orbit_{\tildeG} \cap L $, so that 
$ X, Y \in L $ and $ Y = g \cdot X $ for some $ g \in \tildeG $.  
We want to show $ X $ and $ Y $ are conjugate under the action of $ G $.  
Let us begin with
\begin{align*}
\alpha(g \cdot X) 
&= \sigma( g \cdot X) = \sigma(g)^{-1} \cdot \sigma(X) 
\\
&= \sigma(g)^{-1} \cdot \alpha(X) = \alpha(\sigma(g)^{-1} \cdot X ) , 
\end{align*}
hence we get 
$ g \cdot X = \sigma(g)^{-1} \cdot X $.  
This means 
$ (\sigma(g) g) \cdot X = X $ and we conclude that 
$ \sigma(g) g \in \tildeG_X $ (a stabilizer of $ X $ in $ \tildeG $).
We put $ h := \sigma(g) g $.  
Since $ h \cdot X = X $, we get
\begin{equation*}
X = (u, \transpose{v}) + A 
= h \cdot X 
= (h u, \transpose{(\sigma(h)^{-1} v)}) + h A h^{-1}
= (h u, \transpose{(h^{-1} v)}) + h A h^{-1}.
\end{equation*}
Thus we have 
\begin{equation*}
\sigma(h) = h ; \quad
h u = u, \quad
h v = v, \quad
h A = A h.
\end{equation*}
A routine elementary method in linear algebra tells that 
there exists a polynomial $ f(T) \in \C[T] $ such that 
\begin{align*}
&
h = f(h)^2 , \quad
\sigma(f(h)) = f(h), 
\\
&
f(h) u = u, \quad
f(h) v = v, \quad
f(h) A = A f(h).
\end{align*}
Since $ f(h) $ is regular, it belongs to $ \tildeG_X $.  
We will simply write $ f = f(h) \in \tildeG_X $ below.

Since $ \sigma(f) f = f^2 = h = \sigma(g) g $, we get 
\begin{equation*}
\sigma(f g^{-1}) f g^{-1} 
= \sigma(g)^{-1} \sigma(f) f g^{-1} 
= \sigma(g)^{-1} \sigma(g) g g^{-1} 
= e.
\end{equation*}
This means $ f g^{-1} \in G $ or $ g f^{-1} \in G $.  
Note that $ f^{-1} \in \tildeG_X $ because $ f \in \tildeG_X $.  
Now we finish the proof to see 
\begin{equation*}
Y = g \cdot X = (g f^{-1}) \cdot X \quad 
\text{ and } \quad 
g f^{-1} \in G, 
\end{equation*}
hence $ X $ and $ Y $ are in the same $ G $-orbit.

Let us show that if $ \orbit_G \subset L $ is a closed $ G $-orbit, 
then $ \orbit_{\tildeG} = \tildeG \cdot \orbit_G \subset \tildeL $ is also closed.  
For this, we follow the same arguments as in \cite[Theorem 8]{Ohta.2008}.  
Namely, we use the following criterion of Luna 
(\cite[\S\S~2.1 \& 3.1]{Luna.1975}).  

\medskip

\noindent
\textbf{Luna's Criterion:} 
\textit{Suppose that a reductive group $ G $ acts on an affine variety $ X $ and that 
$ H $ is a reductive subgroup of $ G $.  
Put $ X^H := \{ x \in X \mid h \cdot x = x \} $.  
{\upshape (i)} The natural morphism $ X^H \git N_G(H) \to X \git G $ is finite, 
where $ X \git G = \Spec(\C[X]^G) $ denotes an affine categorical quotient.  
{\upshape (ii)} For any $ x \in X^H $, a $ G $-orbit $ G \cdot x $ is closed 
if and only if $ N_G(H) \cdot x $ is closed.
}

\medskip

In this criterion, we can replace the normalizer $ N_G(H) $ by the centralizer $ Z_G(H) $.
For this, see \cite[\S~1.9]{Luna.1975} 
(cf. \cite[Theorems 6.16 \& 6.17]{Popov.Vinberg.1994} and Remark after Corollary 2).  

In Luna's Criterion, let us take 
$ G \text{ (Luna's)} = \langle \theta \rangle \ltimes \tildeG, \;
X = \tildeL, \; 
H = \langle \theta \rangle \simeq \Z_2 $.  
The action of $ \tildeG $ on $ \tildeL $ is already given, 
and $ \theta $ acts on $ \tildeL $ by 
$ \theta (x) = \alpha^{-1} \sigma_{\tildeL}(x) $ 
for $ x \in \tildeL $.  
Then it is immediate to check that 
$ X^H = L $ and $ Z_G(H) =  \langle \theta \rangle \times G \text{ (ours)}$.

Now Luna's Criterion (ii) tells that 
$ \langle \theta \rangle \cdot \orbit_G $ is closed 
if and only if $ (\langle \theta \rangle \ltimes \tildeG) \cdot \orbit_G $ is closed.  
Since $ \theta $ normalizes $ \tildeG $ and fixes $ \orbit_G $ point wise, 
we know $ \langle \theta \rangle \cdot \orbit_G = \orbit_G $ and 
$ (\langle \theta \rangle \ltimes \tildeG) \cdot \orbit_G = \tildeG \cdot \orbit_G $.  
At the same time, we also see that the natural morphism 
$ L \git G \to \tildeL \git (\langle \theta \rangle \ltimes \tildeG) $ is finite by Criterion (i).
\end{proof}

\section{Orbit embedding in general}
\label{setion:main.theorem}

We extract general conditions from Section~\ref{section:enhanced.Lie.algebra}, which assures orbit embedding.  
Thus, let us consider a general action of an algebraic group $ \tildeG $ on a variety $ \tildeL $, 
which is no more a linear action.  
Let $ \sigma = \sigma_{\tildeG} $ be an anti-automorphism such that $ \sigma_{\tildeG}^2 = \id_{\tildeG} $, and 
assume that there exists an involutive automorphism $ \sigma = \sigma_{\tildeL} \in \Aut(\tildeL) $ which satisfies 
\begin{equation}
\sigma(g \cdot x) = \sigma(g)^{-1} \sigma(x) \qquad
(g \in \tildeG, \; x \in \tildeL).
\end{equation}
Note that $ \theta(g) := \sigma(g)^{-1} $ defines an involutive automorphism of $ \tildeG $.  
We put 
\begin{align*}
G &= \{ g \in \tildeG \mid \sigma(g) = g^{-1} \} = (\tildeG)^{\theta} ,
\\
L &= \{ x \in \tildeL \mid \sigma(x) = x \} \subset \tildeL .
\end{align*}
Then $ G $ is a symmetric subgroup in $ \tildeG $ and $ L $ is a closed subvariety of $ \tildeL $ stable under $ G $.

We assume the following

\begin{assumption}
\label{assumption:antisymmetric.square.in.Gx}
For any $ x \in L $ and $ h \in \tildeG_x $ such that $ \sigma(h) = h $, 
there exists $ f \in \tildeG_x $ which satisfies $ h = f^2 $ and $ \sigma(f) = f $.  
\end{assumption}

\begin{theorem}
\label{theorem:main.theorem.abstract.version}
Under Assumption~{\upshape\ref{assumption:antisymmetric.square.in.Gx}}, 
the following {\upshape (1) -- (3)} hold.
\begin{thmenumerate}
\item
The natural map $ L/G \to \tildeL / \tildeG $ defined by 
$ \orbit_G \mapsto \tildeG \cdot \orbit_G \; (\orbit_G \in L/G) $ induces an injection 
$ L/G \inclusion \tildeL / \tildeG $.  
In other words, 
for any $ \tildeG $-orbit $ \orbit_{\tildeG} \subset \tildeL $, the intersection 
$ L \cap \orbit_{\tildeG} $ is either empty or a single $ G $-orbit.
\item
If $ \tildeL $ is an affine variety, 
then the injection $ L/G \to \tildeL / \tildeG $ 
takes closed $ G $-orbits to closed 
$ \tildeG $-orbits.
\item
If $ \tildeL $ is an affine variety, 
the ring of invariants $ \C[L]^G $ is integral over the restriction of the invariant ring 
$ \C[\tildeL]^{\tildeG}\restrict_L $ to $ L $ and they have the same quotient field.
\end{thmenumerate}
\end{theorem}

\begin{remark}
In \S~\ref{section:enhanced.Lie.algebra}, 
we use an extra automorphism $ \alpha \in \Aut_{\tildeG}(\tildeL) $.  
To recover the situation in \S~\ref{section:enhanced.Lie.algebra}, 
it is enough to consider 
$ \alpha^{-1} \sigma_{\tildeL} $ instead of $ \sigma_{\tildeL} $ here.
\end{remark}

\begin{proof}
The proof of (1) and (2) goes almost the same as that of 
Theorem~\ref{theorem:orbit.embedding.for.enhanced.Lie.algebra} 
(we use Assumption~\ref{assumption:antisymmetric.square.in.Gx} to get an element $ f \in \tildeG_x $).

Let us prove (3). 
Luna's Criterion (ii) tells that 
$ \C[L]^{\langle \theta \rangle \times G} $ 
is of finite type over 
$ \C[\tildeL]^{\langle \theta \rangle \ltimes \tildeG}\restrict_L $.  
Since $ \theta $ acts on $ L $ trivially, 
the invariant ring 
$ \C[L]^G $ coincides with 
$ \C[L]^{\langle \theta \rangle \times G} $.  
On the other hand, the ring 
$ \C[\tildeL]^{\tildeG}\restrict_L $ 
clearly contains 
$ \C[\tildeL]^{\langle \theta \rangle \ltimes \tildeG}\restrict_L $.  
Therefore 
$ \C[L]^G $ is of finite type over 
$ \C[\tildeL]^{\tildeG}\restrict_L $.  

Let us prove their quotient fields are the same.  
In general, we have 
\begin{equation*}
( \closure{\tildeG \cdot L} ) \git \tildeG \simeq \Spec( \C[\tildeL]^{\tildeG}\restrict_L) .
\end{equation*}
See \cite[Proposition 7]{Ohta.2008} for example.  
Thus the natural inclusion 
\begin{equation*}
\varphi^{\ast} : \C[\tildeL]^{\tildeG}\restrict_L
\hookrightarrow \C[L]^G
\end{equation*}
induces a dominant map 
\begin{equation*}
\varphi : L \git G \longrightarrow ( \closure{\tildeG \cdot L} ) \git \tildeG .  
\end{equation*}
The affine quotient $ L \git G $ is in bijection with the set of closed $ G $-orbits in $ L $ via 
the quotient map $ L \to L \git G $.  
Since a closed $ G $-orbit generates a closed $ \tildeG $-orbit, 
the map $ \varphi $ is generically one-to-one.  
This means $ \varphi $ is birational, hence the quotient fields are the same.
\end{proof}

Here we briefly discuss on the application to the invariant theory.  
In the setting of Section~\ref{section:enhanced.Lie.algebra}, 
the invariants of $ \tildeG = \GL(V) $ on 
$ \tildeL = V \oplus V^{\ast} \oplus \gl(V) $ is well known to experts 
(see, e.g., \cite[Theorem 0.2]{Itoh.2013}).

\begin{proposition}
\label{proposition:generators.of.invariants.for.GLV}
$ \GL(V) $-invariants 
on $ V \oplus V^{\ast} \oplus \gl(V) $ are generated by 
invariants on $ \gl(V) $ and $ \{ \Gamma_k \mid k \geq 0 \} $, where 
\begin{equation*}
\Gamma_k(v+\xi+A) = \xi(A^k v) \qquad 
(v + \xi + A \in V \oplus V^{\ast} \oplus \gl(V)).
\end{equation*}
\end{proposition}

Now we apply Theorem~\ref{theorem:main.theorem.abstract.version} (3) to get the following corollary.  

\begin{corollary}
\label{corollary:generators.of.invariants.for.SpV}
$ \Sp(V) $-invariants on $ V \oplus \lie{sp}(V) $ is rationally generated by 
invariants on $ \lie{sp}(V) $ and $ \{ \gamma_k \mid k \geq 1 , \text{odd} \} $, where 
\begin{equation*}
\gamma_k(v+A) = \langle v, A^k v\rangle = \transpose{v} J A^k v \qquad 
(v + A \in V \oplus \lie{sp}(V)).
\end{equation*}
\end{corollary}

\begin{proof}
Since 
$ \Gamma_k $ in Proposition~\ref{proposition:generators.of.invariants.for.GLV} restricted to 
$ V \oplus \lie{sp}(V) $ is $ \gamma_k $, 
they rationally generate the invariants.  
However, 
if $ k $ is even, we have 
$ \langle v, A^k v \rangle = (-1)^{k/2} \langle A^{k/2} v, A^{k/2} v\rangle = 0 $, 
which tells that $ \gamma_k $ vanishes when $ k $ is even.
\end{proof}

Note that $ V \oplus \lie{sp}(V) \simeq V \oplus S^2(V) $ as a 
representation of $ \Sp(V) $.  
Therefore the above corollary also applies to  
invariants $ \C[V \oplus S^2(V)]^{\Sp(V)} \vphantom{\Big|} $.

\section{Further example: enhanced $ \theta $-representation}
\label{sec:further.examples}

Return to the setting and notation in Section~\ref{section:enhanced.Lie.algebra}.
The symplectic space $ V $ has a standard polarization 
$ V = W^+ \oplus W^- $, where $ W^+ $ is generated by the first half standard basis 
$ e_1, \dots, e_n $ and 
$ W^- $ generated by the latter half $ f_1, \dots, f_n $, where $ f_i = e_{n + i} $.  
Let us consider a subgroup $ \tildeK $ of $ \tildeG $:
\begin{equation*}
\tildeK = \Bigl\{ \mattwo{a}{}{}{b} \mid a, b \in \GL_n(\C) \Bigr\} 
\simeq \GL(W^+) \times \GL(W^-).  
\end{equation*}
This is also a symmetric subgroup of $ \tildeG $ associated with an involution 
$ \tau(g) = I_{n,n} g I_{n,n}^{-1} $, where 
$ I_{n,n} = \mattwo{1_n}{}{}{-1_n} $.  
Since $ \tau $ and $ \sigma $ commute with each other, 
$ \tildeK $ is stable under $ \sigma $, and 
it is straightforward to see that $ \sigma(\diag(a, b)) = \diag(\transpose{b}, \transpose{a}) $.  

We can take $ \tildeK $ as $ \tildeG $ and play the same game.  
In this case, we get 
$ K = \{ g \in \tildeK \mid \sigma(g) = g^{-1} \} = (\tildeK)^{\theta} $ in place of $ G $.  
$ K $ is isomorphic to $ \GL_n(\C) $ realized as $ K = \{ \diag(a, \transpose{a}^{-1}) \mid a \in \GL_n(\C) \} $ in 
$ \Sp(V) $.  
Thus our embedding theorem tells us 

\begin{theorem}
\label{theorem:complex.symm.pair.typeA}
Under the same setting and notation as in Section~\ref{section:enhanced.Lie.algebra},
The orbit map $ L/K \to \tildeL / \tildeK $ is an embedding, 
which takes closed orbits to closed orbits.
\end{theorem}

If we take an appropriate subspace of $ \tildeL $ (hence $ L $) stable under $ \tildeK $, 
we get two corollaries.

\begin{corollary}
\label{cor:type.CI.to.AIII.full.version}
For $ V = \C^{2n} $, 
we have an inclusion of orbits:
\begin{equation*}
\bigl( V \oplus \Sym_n(\C) \oplus \Sym_n(\C) \bigr) / \GL_n(\C) \inclusion 
\bigl( V \oplus V^{\ast} \oplus \Mat_n(\C) \oplus \Mat_n(\C) \bigr) / \GL_n(\C) \times \GL_n(\C).
\end{equation*}
\end{corollary}

\begin{proof}
We take a subspace 
\begin{align*}
\tildeL_1 = V \oplus V^{\ast} \oplus (\Mat_n(\C) \oplus \Mat_n(\C)) 
\subset V \oplus V^{\ast} \oplus \Mat_{2n}(\C) = \tildeL ,
\end{align*}
where $ \Mat_n(\C) $ is anti-diagonally embedded into $ \Mat_{2n}(\C) = \lie{gl}_{2n}(\C) $.  
Note that if we denote a Cartan decomposition of $ \lie{gl}_{2n}(\C) $ with respect to $ \tau $ by 
$ \lie{gl}_{2n}(\C) = \tildek \oplus \tildep $, then 
we can identify $ \tildep = \Mat_n(\C) \oplus \Mat_n(\C) $ as above.  

Since $ \sigma = \sigma_{\tildeL} $ preserves $ \tildeL_1 $, 
we take the restriction $ \sigma_1 = \sigma_{\tildeL} \restrict_{\tildeL_1} $ and 
$ \alpha = -1 $.  
Then we get 
\begin{equation*}
L_1 = \{ (v, \transpose{v}) \mid v \in V \} \oplus (\Sym_n(\C) \oplus \Sym_n(\C)) \subset \tildeL_1, 
\end{equation*}
and we can apply Theorem~\ref{theorem:main.theorem.abstract.version} 
(or Theorem~\ref{theorem:complex.symm.pair.typeA} directly).
\end{proof}

\begin{corollary}
For $ W = \C^{n} $, 
we have an inclusion of orbits:
\begin{equation*}
\label{corollary:enhanced.Kronecker.quiver}
\bigl( W \oplus \Sym_n(\C) \oplus \Sym_n(\C) \bigr) / \GL_n(\C) \inclusion 
\bigl( W \oplus W^{\ast} \oplus \Mat_n(\C) \oplus \Mat_n(\C) \bigr) / \GL_n(\C) \times \GL_n(\C).
\end{equation*}
\end{corollary}

\begin{proof}
We take a subspace $ \tildeL_2 $ in $ \tildeL_1 $ given in 
the proof of Corollary~\ref{cor:type.CI.to.AIII.full.version} as follows.  

Let $ W = W^+ $ be a Lagrangian subspace in $ V = \C^{2n} $ as explained above.  
Then $ W^{\ast} $ can be naturally identified with $ W^- $ by the symplectic form on $ V $.  
However, in our notation 
$ X = (u, \transpose{v}) + A \in V \oplus V^{\ast} \oplus \End(V) $, 
the second entry $ \transpose{v} $ should have been understood as $ \transpose{v} J $ 
if we choose a non-degenerate bilinear form on $ V $ as the symplectic form 
$ \langle v, u \rangle = \transpose{v} J u \; (u, v \in V) $.  
But we did not do it in that way to avoid the complexity of the notation.  
Therefore, $ W^{\ast} = W^- $ should be interpreted as 
$ \transpose{W^-} \cdot J = \transpose{W^+} $, 
which is the standard way to identify $ \C^n $ with $ (\C^n)^{\ast} $.  
Thus our $ W \oplus W^{\ast} $ is expressed in the coordinate 
\begin{equation*}
W \oplus W^{\ast} 
= \{ ((u, 0), \transpose{(v, 0)}) \mid u, v \in \C^n \} \subset V \oplus V^{\ast}.
\end{equation*}
With this understood we can see that $ \tildeL_2 $ is stable under $ \sigma = \sigma_{\tildeL} $ and we get 
\begin{align*}
L_2 
&= \{ ((u, 0), \transpose{(u, 0)}) \mid u \in \C^n \} \oplus (\Sym_n(\C) \oplus \Sym_n(\C))
\\
&= W \oplus (\Sym_n(\C) \oplus \Sym_n(\C)) \subset L_1 .
\end{align*}
Now everything follows.
\end{proof}

Corollary~\ref{corollary:enhanced.Kronecker.quiver} can be considered as an orbit embedding theorem 
for an \emph{enhanced cyclic quiver} 
(see \cite{Ohta.2010} and \cite{Kempken.1982} for cyclic quivers, 
and \cite{Johnson.2010} for enhanced cyclic quiver), 
and it seems deeply related to enhanced (or exotic) nilpotent cone for 
a double flag variety of type CI symmetric space (in preparation).

\section{Obstructions for orbit embedding}

In this section, we will exhibit a simple example which indicates what is an obstruction 
for orbit embedding.

In this example, we use traditional notation, which is different from the sections above.  
Let $ G = \GL_2(\C) $ which acts on $ \tildeX = \GL_2(\C) $ itself by conjugation: 
$ g \cdot x = g x g^{-1} \; (g \in G, \, x \in \tildeX) $.  
We define an involutive anti-automorphism $ \sigma : G \to G $ by 
$ \sigma(g) = I_{1,1}^{-1} g^{-1} I_{1,1} $, where 
$ I_{1,1} = \mattwo{1}{0}{0}{-1} $.  
Then clearly $ \sigma(g \cdot x) = \sigma(g)^{-1} \cdot \sigma(x) $ holds.  
Define an involutive automorphism $ \theta \in \Aut(G) $ by 
$ \theta(g) = \sigma(g)^{-1} = I_{1,1}^{-1} g I_{1,1} $, and put
\begin{align*}
K &= G^{\theta} = \{ g \in G \mid \sigma(g)= g^{-1} \} = \{ g \in G \mid \theta(g) = g \} , 
\\
X &= G^{\sigma} = \{ g \in G \mid \sigma(g) = g \} .
\end{align*}
By direct calculation, we know
\begin{align*}
K &= \Bigl\{ \mattwo{s}{0}{0}{t} \bigm| s t \neq 0 \Bigr\}
\qquad 
\text{ and }
\\
X &= \Bigl\{ \mattwo{a}{b}{c}{a} \bigm| a^2 - bc = 1, \; b \neq 0 \text{ or } c \neq 0 \Bigr\} 
\cup \Bigl\{ \mattwo{a}{0}{0}{d} \bigm| a^2 = d^2 = 1 \Bigr\} .
\end{align*}
We want to analyze the map $ X/K \to \tildeX / G $.  
This is almost an inclusion but not quite.

After simple computations, it turns out that a complete set of representatives of $ K $-orbits in $ X $ is 
\begin{align*}
& \Bigl\{ \mattwo{a}{b}{b}{a} \bigm| a^2 - b^2 = 1, \; \Re b > 0 \text{ or } b \in i \R_{\geq 0} \Bigr\} 
\cup 
\\
& \Bigl\{ \mattwo{a}{1}{0}{a} \bigm| a = \pm 1 \Bigr\} 
\cup \Bigl\{ \mattwo{a}{0}{1}{a} \bigm| a = \pm 1 \Bigr\} 
\cup \Bigl\{ \mattwo{a}{0}{0}{-a} \bigm| a = \pm 1 \Bigr\} .
\end{align*}
Among them, the following is the complete list of pairs of 
mutually different representatives of $ K $-orbits, which generates the same $ G $-orbits.  
There are only three of such pairs (and no triples).
\begin{align*}
\Bigl\{ \mattwo{a}{1}{0}{a} , \mattwo{a}{0}{1}{a} \Bigr\} \;\; (a = \pm 1), 
\qquad
\Bigl\{ \mattwo{1}{0}{0}{-1} , \mattwo{-1}{0}{0}{1} \Bigr\} .
\end{align*}
For example, if we take 
$ x = \mattwo{1}{1}{0}{1} $, then the stabilizer is 
\begin{equation*}
G_x = \Bigl\{ \mattwo{\alpha}{\beta}{0}{\alpha} \bigm| \alpha \neq 0 \Bigr\} 
\qquad 
\text{ and }
\qquad
G_x^{\sigma} = \Bigl\{ \mattwo{\alpha}{\beta}{0}{\alpha} \bigm| \alpha = \pm 1 \Bigr\} .
\end{equation*}
Now we see that for $ h = \mattwo{-1}{\beta}{0}{-1} \in G_x^{\sigma} $, 
there is no $ f \in G_x^{\sigma} $ such that $ h = f^2 $.  
Similarly, if we take $ x = \mattwo{1}{0}{0}{-1} $, then 
$ G_x = K $ and 
\begin{align*}
G_x^{\sigma} &= \Bigl\{ \mattwo{\alpha}{0}{0}{\delta} \bigm| \alpha^2 = \delta^2 = \pm 1 \Bigr\} .
\end{align*}
Again, for $ h = - 1_2 $ or $ \pm \mattwo{1}{0}{0}{-1} \in G_x^{\sigma} $, 
there exists no $ f \in G_x^{\sigma} $ which satisfies $ h = f^2 $.  
These are obstructions which prevent the map $ X/K \to \tildeX / G $ 
from being injective.

\renewcommand{\MR}[1]{}

\def\cftil#1{\ifmmode\setbox7\hbox{$\accent"5E#1$}\else
  \setbox7\hbox{\accent"5E#1}\penalty 10000\relax\fi\raise 1\ht7
  \hbox{\lower1.15ex\hbox to 1\wd7{\hss\accent"7E\hss}}\penalty 10000
  \hskip-1\wd7\penalty 10000\box7} \def\cprime{$'$} \def\cprime{$'$}
  \def\Dbar{\leavevmode\lower.6ex\hbox to 0pt{\hskip-.23ex \accent"16\hss}D}
\providecommand{\bysame}{\leavevmode\hbox to3em{\hrulefill}\thinspace}
\providecommand{\MR}{\relax\ifhmode\unskip\space\fi MR }
\providecommand{\MRhref}[2]{%
  \href{http://www.ams.org/mathscinet-getitem?mr=#1}{#2}
}
\providecommand{\href}[2]{#2}

\end{document}